\documentclass[12pt]{article}
\usepackage[french, english]{babel}
\usepackage[T1]{fontenc}
\usepackage[utf8]{inputenc}
\usepackage{amsmath,amssymb,latexsym,amsthm}
\usepackage[shortlabels]{enumitem}
\usepackage[unicode]{hyperref}
\usepackage[abbrev]{amsrefs}
\usepackage{tikz}
\usetikzlibrary{intersections}
\usepackage{microtype}
\usepackage{geometry}

\theoremstyle{plain}
\newtheorem{theorem}{Theorem}[section]
\newtheorem{proposition}[theorem]{Proposition}
\newtheorem{corollary}[theorem]{Corollary}
\newtheorem{lemma}[theorem]{Lemma}

\newtheorem*{mainthm}{Main Theorem}

\newtheorem*{joinlemma}{Join Lemma}
\newtheorem*{RAAG conjecture}{Action Dimension Conjecture for RAAGs}
\newtheorem*{Aconj}{Action Dimension Conjecture}
\newtheorem*{Sconj}{Singer Conjecture}

\newtheorem*{cor}{Corollary} 

\theoremstyle{definition}

\theoremstyle{remark}
\newtheorem{remark}[theorem]{Remark}
\newtheorem*{rmk}{Remark}

\newtheorem{examples}[theorem]{Examples}

\newcommand{\wt}{\widetilde}
\newcommand{\ga}{\alpha}

\newcommand{\geps}{\varepsilon}
\newcommand{\gs}{\sigma}
\newcommand{\gt}{\tau}
\newcommand{\gG}{\Gamma}
\newcommand{\gD}{\Delta}
\newcommand{\gO}{\Omega}

\newcommand{\gS}{\Sigma}
\newcommand{\pinfty}{\partial_\infty}
\newcommand{\btwo}{b^{(2)}}
\newcommand{\cat}{\operatorname{CAT}}
\newcommand{\Hom}{\operatorname{Hom}}
\newcommand{\lk}{\operatorname{Lk}}
\newcommand{\st}{\operatorname{St}}
\newcommand{\vk}{\operatorname{vk}}
\newcommand{\vkt}{\vk_{\zz/2}}

\newcommand{\edge}{\operatorname{Edge}}

\newcommand{\obdim}{\operatorname{obdim}}
\newcommand{\vkdim}{\operatorname{vkdim}}
\newcommand{\actdim}{\operatorname{actdim}}
\newcommand{\embdim}{\operatorname{embdim}}
\newcommand{\ltdim}{\operatorname{\ell^2dim}}
\newcommand{\geod}{\operatorname{gd}}

\newcommand{\redh}{\overline{H}}
\newcommand{\redb}{\overline{b}}
\newcommand{\minus}{^{-1}}
\newcommand{\lb}{b^{(2)}}
\newcommand{\OL}{{OL}}
\newcommand{\XLT}{\wt{X}_L}
\newcommand{\cac}{\mathcal{C}}
\newcommand{\QQ}{\mathbb{Q}}
\newcommand{\rr}{\mathbb{R}}
\newcommand{\zz}{\mathbb{Z}}

\newenvironment{enumerate1}{
\begin{enumerate}[\upshape (1)]}
	{
\end{enumerate}
}
\newenvironment{enumeratea}{
\begin{enumerate}[\upshape(a)]}
	{
\end{enumerate}
}

\begin{document}

\title{The action dimension of right-angled Artin groups}

\author{Grigori Avramidi\thanks{The first author was partially supported by an NSF grant.}
\and Michael W. Davis\thanks{The second author was partially supported by an NSF grant.}
\and Boris Okun\thanks{The third author was partially supported by a Simons Foundation Collaboration Grant.
He also wants to thank the Max Planck Institute for Mathematics in Bonn for hospitality.}
\and Kevin Schreve }

\date{\today} \maketitle
\begin{abstract}
	\noindent The \emph{action dimension} of a discrete group $\Gamma$ is the smallest dimension of a contractible manifold which admits a proper action of $\gG$.
	Associated to any flag complex $L$ there is a right-angled Artin group, $A_L$.
	We compute the action dimension of $A_L$ for many $L$.
	Our calculations come close to confirming the conjecture that if an $\ell^2$-Betti number of $A_L$ in degree $l$ is nonzero, then the action dimension of $A_L$ is $\geq 2l$.
	\smallskip
	
	\noindent \textbf{AMS classification numbers}.
	Primary: 57Q15, 57Q25, 20F65,\\
	Secondary: 57R58 \smallskip
	
\end{abstract}

\section*{Introduction}
If a group $\gG$ has a finite dimensional classifying space $B\gG$, then its \emph{geometric dimension}, denoted $\geod\gG$, is the minimum dimension of a model for $B\gG$.
Its \emph{action dimension}, denoted $\actdim \gG$, is the minimum dimension of a contractible manifold $M$ which admits a proper $\gG$-action.
If $\gG$ is torsion-free, then any proper $\gG$-action is free; so, $M/\gG$ is a finite dimensional model for $B\gG$.

Any $k$-dimensional simplicial complex embeds in general position in $\rr^{2k+1}$.
So, if $\dim B\gG = k$, then $\actdim \gG\le 2k+1$.
(A regular neighborhood of $B\gG$ in $\rr^{2k+1}$ is an aspherical manifold; its universal cover is a contractible manifold on which $\gG$ acts properly.) This estimate can be improved by 1 since any $k$-dimensional CW complex is homotopy equivalent to a $2k$-manifold with boundary.
(Replace the cells by $2k$-dimensional handles --- by general position the handle attaching maps can be chosen to be embeddings.) Hence, if $\geod\gG=k$, then some model for $B\gG$ can be thickened to an aspherical manifold of dimension $2k$.
Thus,
\[
	\actdim \gG \le 2\geod\gG .
\]
(Alternatively, by a theorem of Stallings \cite{stallings}, any $k$-complex is homotopy equivalent to another $k$-complex which embeds in $\rr^{2k}$, cf.~\cite{dr93}).

In \cite{bkk} Bestvina, Kapovich and Kleiner introduced a method for determining $\actdim\gG$.
This method relates the action dimension to the minimum dimension $m$ in which certain finite simplicial complexes can embed piecewise linearly in $S^m$.
For a given finite simplicial complex $K$, this $m$ is called the \emph{embedding dimension} of $K$ and is denoted $\embdim K$.
The technique of \cite{bkk} is based on the mod 2 van Kampen obstruction to embedding $K$ into $S^m$.
We review this obstruction.
Let $\cac (K)$ denote the configuration space of unordered pairs of distinct points in $K$, i.e., if $\gD$ denotes the diagonal in $K\times K$, then $\cac(K)$ is the quotient of $(K\times K) -\gD$ by the involution which switches the factors.
The double cover $(K\times K) -\gD\to \cac (K)$ is classified by a map $c:\cac(K)\to \rr P^\infty$.
The \emph{van Kampen obstruction} in degree $m$ is the cohomology class $\vkt^m(K)\in H^m(\cac(K);\zz/2)$ defined by
\[
	\vkt^m(K) = c^*(w_1^m),
\]
where $w_1\in H^1(\rr P^\infty;\zz/2)$ is the first Stiefel--Whitney class of the canonical line bundle over $\rr P^\infty$.
The class $\vkt^m(K)$ is an obstruction to embedding $K$ in $S^m$.
We say $K$ is an \emph{$m$-obstructor} if $\vkt^m(K)\neq 0$.
The \emph{van Kampen dimension} of $K$, denoted by $\vkdim K$, is the maximum $m$ such that $\vkt^m (K)\neq 0$.
Thus, $\vkdim K +1\leq \embdim K$.
In a similar fashion, one defines an integral version of the van Kampen obstruction, denoted $\vk^m(K)$.
(We will recall its definition in Section~\ref{s:vk}.) When $m=2\dim K$ and $\dim K \neq 2$, the cohomology class $\vk^m(K)$ is the complete obstruction to embeddability; moreover, its nontriviality is often detected by the mod 2 version.
So, in many cases, $\vkdim K +1 = \embdim K$.
(It is shown in \cite{fkt94} that the integral van Kampen obstruction is incomplete for $\dim K=2$, i.e., there is a $2$-complex $K$ with $\vk^4(K)=0$, yet $\embdim K=5$.)

To fix ideas, suppose that $\gG$ is of type $F$ and that $E\gG$, the universal cover of $B\gG$, has a $Z$-set compactification.
Denote the boundary of this compactification by $\pinfty \gG$.
Suppose further that $\gG$ acts properly on a contractible $n$-manifold $M$ which has a $Z$-set compactification with boundary $\pinfty M$ and that the equivariant map $E\gG\to M$ extends to an inclusion of $Z$-set boundaries.
(For example, this is the case, if $M$ is a proper $\cat(0)$-space and $E\gG$ is a convex subspace.) To further simplify the discussion, suppose $\pinfty M$ is homeomorphic to $S^{n-1}$.
If $K$ is a finite complex embedded in $\pinfty \gG$, then $K\subset \pinfty \gG\subset \pinfty M=S^{n-1}$.
So, one expects \( \actdim \gG \geq \embdim K +1. \) Since $\embdim K\geq \vkdim K + 1$, this would entail
\[
	\actdim \gG \geq \vkdim K +2.
\]
Roughly, the definition of \cite{bkk} of the \emph{obstructor dimension of $\gG$}, denoted $\obdim \gG$, is the maximum of $\vkdim K +2$, where $K$ ranges over the finite subcomplexes of $\pinfty \gG$.
The actual definition of $\obdim\gG$ in \cite{bkk}*{p.~225} does not depend on the choice of model for $E\gG$ and does not require a $Z$-set compactification.
Moreover, $\obdim \gG$ provides a lower bound for $\actdim \gG$, cf., \cite{bkk}*{Section 3}.

A particularly tractable case to which the theory of \cite{bkk} can be applied is when $\gG=A_L$, the right-angled Artin group (abbreviated RAAG) associated to a finite flag complex $L$.
The standard classifying space $BA_L$ for $A_L$ is a subcomplex of a torus which has one $S^1$ factor for each vertex of $L$ (cf.~Section~\ref{s:prelim}).The space $BA_L$ is a locally $\cat(0)$ cube complex of dimension equal to $\dim L +1$.
Since this is the cohomological dimension of $A_L$
we have $\geod A_L = \dim L +1$.

The link of a vertex in $BA_L$ is a certain finite simplicial complex $\OL$ called the \emph{octahedralization} of $L$.
The complex $\OL$ is constructed by ``doubling the vertices of $L$.'' When $L$ is a $k$-simplex, $\OL$ is the boundary complex of a $(k+1)$-octahedron (cf.
Section~\ref{s:prelim}).
We shall see in Section~\ref{s:prelim} that $\OL \subset \pinfty A_L$.

We will show in Proposition~\ref{p:emb-act} that if $\OL$ piecewise linearly embeds in $S^m$ (and if the codimension is not $2$), then $A_L$ acts on a contractible $(m+1)$-manifold, i.e.,
\begin{equation*}
	\embdim \OL +1\geq \actdim A_L\geq \vkdim \OL +2.
\end{equation*}

Our main result concerns $\vkdim \OL$ for a flag complex $L$.
\begin{mainthm}
	Suppose $L$ is a $k$-dimensional flag complex.
	\begin{enumerate1}
		\item If $H_k(L;\zz/2)\neq 0$, then $\vkdim \OL= 2k$.
		Consequently,
		\[
			\actdim A_L =2k+2= 2\geod A_L.
		\]
		
		\item If $H_k(L;\zz/2)= 0$, then $\vk^{2k}(\OL)=0$.
		So, for $k\neq 2$, $\embdim \OL \leq 2k$.
		Consequently, \( \actdim A_L\leq 2k+1. \)
	\end{enumerate1}
\end{mainthm}
\begin{cor}
	Suppose $\dim L = k$ with $k\neq 2$.
	Then $A_L$ is the fundamental group of an aspherical $(2k+1)$-manifold if and only if $H_k(L;\zz/2)=0$.
\end{cor}
\begin{rmk}
	For $k=1$, the corollary was proved previously by Droms \cite{d87a}.
	In \cite{gordon} Gordon extended this to all Artin groups as follows: Suppose $L$ is the nerve of a Coxeter group where the edges of $L$ are labeled by integers $\geq 2$ and that $A_L$ is the corresponding Artin group.
	Then $A_L$ is a $3$-manifold group if and only if each component of $L$ is either a tree or a $2$-simplex with edges labeled $2$.
	(In the case where all edge labels of $L$ are required to be even, this had been proved earlier by Hermiller and Meier \cite{hm99}.)
\end{rmk}

The $(k+1)$-fold direct product of nonabelian free groups is a RAAG to which part (1) of the Main Theorem can be applied.
The corresponding flag complex $L$ is a $(k+1)$-fold join of finite sets, each of which has at least two elements; hence, $\dim L=k$ and $H_k(L;\zz/2)\neq 0$.
In this case, van Kampen showed that $\vkt^{2k}( \OL) \neq 0$ and our Main Theorem already was stated and proved in \cite{bkk}.

The $\ell^2$-Betti numbers $\lb_i(\gG$) are well-defined invariants of a group $\gG$.
The \emph{$\ell^2$-dimension of $\gG$}, denoted $\ltdim \gG$, is defined by
\[
	\ltdim \gG:=\sup\{i\mid \lb_i(\gG)\neq 0\}.
\]
In \cite{do01} the second and third authors conjectured that $\ell^2$-Betti numbers of a group $\Gamma$ should give lower bounds for its action dimension.
More precisely, we have the following.
\begin{Aconj}[Davis--Okun] $\actdim \gG \geq 2\ltdim \gG.$
\end{Aconj}
The $\ell^2$-Betti numbers of $A_L$ were computed by Davis--Leary in \cite{dl03} as follows:
\[
	\btwo_{i+1} (A_L)=\redb_i(L),
\]
where $\redb_i(L)$ denotes the ordinary reduced Betti number, $\dim_{\QQ} \redh_i(L;\QQ)$.
This gives the following corollary to the Main Theorem.
\begin{cor}[cf.~Theorem~\ref{t:actdimconj}] For a $k$-dimensional flag complex $L$, if $H_k(L;\zz/2)\neq 0$, then the Action Dimension Conjecture holds for $A_L$.
\end{cor}

Theorem~\ref{t:actdimconj} provides strong evidence that if $\redb_l(L) \ne 0$, then $\vkt^{2l}(\OL)\ne 0$ which would imply that $\actdim A_L\ge 2\ltdim A_L$.
In other words, Theorem~\ref{t:actdimconj} comes close to providing a proof of the Action Dimension Conjecture for general RAAGs.

\numberwithin{equation}{section}

\section{Octahedralization and RAAGs}\label{s:prelim}
\paragraph{The octahedralization of a simplicial complex} 
Given a finite set $V$, let $\gD(V)$ denote the full simplex on $V$ and let $O(V)$ denote the boundary complex of the octahedron on $V$.
In other words, $O(V)$ is the simplicial complex with vertex set $V\times \{\pm 1\}$ such that a subset $\{(v_0, \geps_0),\dots, (v_k, \geps_k)\}$ of $V\times \{\pm1\}$ spans a $k$-simplex if and only if its first coordinates $v_0,\dots v_k$ are distinct.
Projection onto the first factor $V\times \{\pm1\}\to V$ induces a simplicial projection $p:O(V)\to \gD(V)$.
We denote the vertex $(v,+1)$ or $(v,-1)$ by $v^+$ or $v^-$, respectively.

Any finite simplicial complex $L$ with vertex set $V$ is a subcomplex of $\gD(V)$.
The \emph{octahedralization} $\OL$ of $L$ is the inverse image of $L$ in $O(V)$:
\[
	\OL:= p\minus (L) \subset O(V).
\]
(This terminology comes from \cite{do12}*{Section 8}.) We also say that $\OL$ is the result of ``doubling the vertices of $L$.''

\paragraph{Right-angled Artin groups and right-angled Coxeter groups} 
Suppose $L^1$ is a simplicial graph with vertex set $V$.
The \emph{flag complex determined by} $L^1$ is the simplicial complex $L$ whose simplices are the (vertex sets of) complete subgraphs of $L^1$.
Associated to $L^1$ there is a RAAG, $A_L$, and a right-angled Coxeter group (abbreviated RACG), $W_L$.
These groups are defined by presentations as follows.
A set of generators for $A_L$ is $\{g_v\}_{v\in V}$; there are relations $[g_v, g_{v'}] = 1$ (i.e., $g_v$ and $g_{v'}$ commute) whenever $\{v,v'\}\in \edge L^1$.
The RACG $W_L$ is the quotient of $A_L$ formed by adjoining the additional relations $(g_v)^2=1$, for all $v\in V$.
(Usually, we denote the image of a generator in $W_L$ by $s_v$ rather than $g_v$.)

Let $T^V$ denote the product $(S^1)^V$.
Each copy of $S^1$ is given a (cubical) cell structure with one vertex $e_0$ and one edge.
For each simplex $\gs\in L$, $T(\gs)$ denotes the subset of $T^V$ consisting of those points $(x_v)_{v\in V}$ such that $x_v=e_0$ whenever $v$ is not a vertex of $\gs$.
So, $T(\gs)$ is a standard subtorus of $T^V$; its dimension is $\dim \gs +1$.
The \emph{standard classifying space} for $A_L$ is the subcomplex $X_L$ of $T^V$ defined as the union of the subtori $T(\gs)$ over all simplices $\gs$ in $L$:
\[
	X_L:= \bigcup_{\gs\in L} T(\gs).
\]
The space $X_L$ is sometimes called the ``Salvetti complex.'' Its $2$-skeleton is the presentation complex for $A_L$; so, $\pi_1(X_L)=A_L$.
There is a natural cubical cell structure on $X_L$ with a cube of dimension $\dim \gs +1$ for each $\gs \in L$.
The link of the 0-cell in $X_L$ is $\OL$.
We note that $\OL$ is also a flag complex.
So, the induced cubical structure on the universal cover $\XLT$ is $\cat(0)$.
Hence, $\XLT$ is contractible, i.e., $X_L$ is a model for $BA_L$.
(For more details on $X_L$, see \cite{cd95}*{Section 3}.) Choose a base point $b\in EA_L$ ($=\XLT$) which is a lift of the $0$-cell.
Following \cite{BB}*{Section 6} define a \emph{sheet} in $EA_L$ to be a component of preimage of a standard subtorus.
Let $\rr(\gs)$ be the sheet corresponding to $T(\gs)$ which contains $b$.
Then $Y:=\bigcup_{\gs\in L} \rr(\gs)$ is a convex subcomplex of $EA_L$ isometric to the Euclidean cone on $\OL$.
This gives an embedding of $\OL$ in $\pinfty A_L$.

For any flag complex $K$ there is a standard $\cat(0)$ cubical complex $\gS_K$ (sometimes called the ``Davis complex'') on which the right-angled Coxeter group $W_K$ acts as a cocompact reflection group.
The cubical complexes $\gS_\OL$ and $\XLT$ are identical (cf.
\cite{dj00}); moreover, $W_\OL$ and $A_L$ have a common subgroup, which is of finite index in each.
Also, as is shown by Hsu and Wise in \cite{hw99} there is an embedding $A_L\hookrightarrow W_\OL$ (usually as a subgroup of infinite index) defined by $g_v\mapsto s_{v+}s_{v-}$\,, where $s_{v+}$ and $s_{v-}$ are the generators of $W_\OL$ corresponding to the vertices $v^{+}$ and $v^{-}$, respectively.
\begin{proposition}\label{p:ol}
	Suppose $\OL$ embeds as a full subcomplex of a flag triangulation of $S^m$.
	Then $\actdim A_L \leq m+1$.
\end{proposition}
\begin{proof}
	Let $K$ be the flag triangulation of $S^m$.
	Then $\gS_K$ is a (contractible) $(m+1)$-manifold (cf.~\cite{davisbook}*{Theorem~10.6.1}).
	Since $A_L \subset W_\OL\subset W_K$, $A_L$ acts freely and properly on $\gS_K$.
\end{proof}

To implement this proposition we need the method of ``partial barycentric subdivision,'' which is explained in the following lemma.
\begin{lemma}\label{l:partial-bary}
	Suppose a flag complex $L$ is a subcomplex of another simplicial complex $K$.
	Then there is a subdivision $K'$ of $K$ such that
	\begin{enumeratea}
		\item $L$ is a full subcomplex of $K'$ and
		\item $K'$ is a flag complex.
	\end{enumeratea}
\end{lemma}
\begin{proof}
	The complex $K'$ will be called the \emph{partial barycentric subdivision of $K$ relative to $L$}.
	For each $i\ge 0$, let $K^{(i)}$ denote the set of $i$-simplices in $K$.
	The complex $K'$ will have a new vertex $v_\gs$ for each $\gs\in K^{(i)} - L^{(i)}$, with $i>0$.
	Define the skeleta of $K'$ by induction on dimension.
	First, subdivide each edge in $K^{(1)} - L^{(1)}$ by introducing a midpoint.
	Suppose, by induction, that $i\ge 2$ and that the $(i-1)$-skeleton of $K'$ has been defined.
	Let $\gs\in K^{(i)}$.
	If the $1$-skeleton of $\partial \gs$ lies in $L$, then, since $L$ is flag, $\gs\in L$ and we leave it unchanged.
	If $\gs\in K^{(i)} - L^{(i)}$, then, by inductive hypothesis, $(\partial \gs)'$ has been defined.
	Define $(\gs)'$ to be the result of coning $(\partial \gs)'$ to $v_\gs$.
	It is then easily checked that $K'$ has properties (a) and (b).
\end{proof}

\section{Extending triangulations} 
In \cite{akin} Akin proved the following.
\begin{theorem}[Akin \cite{akin}*{\S VII, Cor.~3, p.~468}]\label{t:akin}
	Suppose $(X,X_0)$ is a locally unknotted polyhedral pair.
	Then any triangulation of $X_0$ extends to a triangulation of $X$.
\end{theorem}
 \noindent (The notion of ``local unknottedness'' is also defined in \cite{akin}*{p.~414}.) 
A corollary to Theorem~\ref{t:akin} is the following.
\begin{proposition}\label{p:emb-act}
	Suppose $L$ is a flag complex and that $\embdim \OL > \dim L+2$.
	Then
	\[
		\actdim A_L\leq \embdim \OL + 1.
	\]
\end{proposition}
\begin{proof}
	Suppose $\OL$ piecewise linearly embeds in $S^m$.
	According to \cite{akin}*{Corollary 9b, p.
	454--455} for $m > \dim L+2$, the embedding is locally unknotted.
	By Theorem~\ref{t:akin}, $\OL$ extends to a triangulation $K$ of $S^m$.
	By Lemma~\ref{l:partial-bary} we can assume that $K$ is a flag complex and that $\OL$ is a full subcomplex.
	The claim now follows from Proposition~\ref{p:ol}.	
\end{proof}

\section{Evaluating the van Kampen obstruction}\label{s:vk}
Given a simplicial complex $K$, from now on $(K\times K)-\gD$ will denote the \emph{simplicial deleted product}, i.e., the union of cells of the form $\gs \times \gt$ where $\gs$ and $\gt$ are (closed) simplices in $K$ and $\gs\cap \gt = \emptyset$.
The \emph{configuration space} $\cac(K)$ is the quotient of $(K\times K)-\gD$ by the involution which switches the factors.
(The simplicial deleted product is an equivariant deformation retract of the actual complement of the diagonal.) The unoriented cells of $\cac(K)$ are represented by unordered pairs $\{ \gs, \gt \}$ of disjoint simplices of $K$.
Note that switching the factors of a cell in $K\times K$ changes orientation by a factor $(-1)^{\dim \gs \dim \gt}$.
To account for this we will represent an oriented cell by an equivalence class $[ \gs, \gt ]$ of ordered pairs $(\gs, \gt)$ of oriented simplices where the equivalence relation is defined by $( \gt, \gs ) \sim (-1)^{\dim \gs \dim \gt}( \gs, \gt )$.

Let $c:\cac(K)\to \rr P^\infty$ be the map which classifies the double cover.

There are two actions of the $\pi_1(\rr P^\infty)$ ($=\zz/2)$ on $\zz$: the trivial action and the nontrivial action where the generator of $\zz/2$ acts by $-1$.
The trivial and nontrivial $\zz/2$-module structures on $\zz$ will be denoted by $\zz^+$ and $\zz^-$, respectively.
Of course,
\[
	H^i(\rr P^\infty;\zz^+)=
	\begin{cases}
		\zz, &\text{if $i=0$;}\\
		\zz/2, &\text{if $i>0$ and is even;}\\
		0, &\text{if $i$ is odd.}
	\end{cases}
\]
The coefficient sequence $0\to \zz^+\to \zz[\zz/2]\to \zz^-\to 0$ induces a long exact sequence in cohomology and since $H^*(\rr P^\infty;\zz[\zz/2])$ vanishes in positive degrees,
\[
	H^i(\rr P^\infty;\zz^-)=
	\begin{cases}
		\zz/2, &\text{if $i$ is odd;}\\
		0, &\text{if $i$ is even.}
	\end{cases}
\]

The Euler class $e_1$ of the canonical line bundle over $\rr P^\infty$ is the nonzero element of $H^1(\rr P^\infty; \zz^-) \cong \zz/2$.
Moreover, if $\geps$ denotes the sign of $(-1)^m$, then $e_1^m$ is the nonzero element of $H^m(\rr P^\infty; \zz^\geps)\cong \zz/2$.
The \emph{integral van Kampen obstruction} $\vk^m(K)$ in degree $m$ is the element of $H^m(\cac(K);\zz^\geps)$ defined by
\[
	\vk^m(K) = c^*(e_1^m).
\]
N.B. Since $e_1^m$ has order $2$, $\vk^m(K)$ has order $\leq 2$.

In what follows we shall be concerned almost exclusively with the case where $m$ is even (so that the coefficients are untwisted).

Suppose $\dim K = k$ and that we want to evaluate the van Kampen obstruction in the top degree $m=2k$.
In \cite{mtw11}*{Appendix D} we find the following description of a cocycle $\nu$ representing the integral van Kampen obstruction $\vk^{2k}(K) \in H^{2k}(\cac(K);\zz)$.
First choose a total ordering, $<$, of the vertices of $K$.
Suppose $\gs=[v_0,\dots, v_k]$ and $\gt=[w_0,\dots, w_k]$ are $k$-simplices with their vertices in increasing order.
Then the value of $\nu$ on the oriented $2k$-cell $[ \gs,\gt ]$ is given by
\begin{equation}\label{e:vktr}
	\nu ([ \gs,\gt ]) =
	\begin{cases}
		+1, &\text{if $v_0<w_0<\cdots <v_k<w_k$,}\\
		(-1)^{k}, &\text{if $w_0<v_0<\cdots <w_k<v_k$,}\\
		\ \ 0, &\text{otherwise.}
	\end{cases}
\end{equation}
(The second clause agrees with our convention on switching factors since $(-1)^{k}=(-1)^{k^{2}}$.) We will say that $\gs$ and $\gt$ are \emph{meshed} if their vertices satisfy the above relationships for $\nu([\gs,\gt]) \neq 0$.

Reducing modulo $2$ gives a cocycle representative for $\vkt^{2k}(K) \in H^{2k}(\cac(K);\zz/2)$:
\begin{equation*}
	\nu_2( \{\gs,\gt\})=
	\begin{cases}
		1, &\text{if $\gs$ and $\gt$ are meshed}\\
		0, &\text{otherwise.}
	\end{cases}
\end{equation*}
\begin{remark}
	The formulas for these cocycles are surprisingly concrete --- they come from using the total ordering to embed $K$ in $\rr^{2k}$ by mapping linearly to the moment curve in $\rr^{2k}$.
	Specifically, if $\gamma(t) = (t, t^2, \dots, t^{2k}) \in \rr^{2k}$, the mapping is determined by sending the $i^{th}$ ordered vertex of $K$ to $\gamma(i)$.
	It turns out that this is a general position map and that the intersections are given by \eqref{e:vktr}.
\end{remark}

\section{Some technical lemmas}

This section contains four lemmas, which we will use in the next section to determine whether or not $\vk(OL)$ vanishes.

\paragraph{The map s.}
There is a transfer map $t: C_{*}(\cac(\OL)) \to C_{*}(\OL \times \OL)$ defined by $[\gs, \gt] \mapsto (\gs,\gt) +(-1)^{\dim \gs \dim \gt}(\gt,\gs)$.
Composing with $p:\OL\to L$ in the second coordinate gives a chain map $s:C_{*}(\cac(\OL)) \to C_{*}(\OL \times L)$ defined by
\[
	s:[\gs, \gt] \mapsto (\gs,p(\gt))+(-1)^{\dim \gs \dim \gt}(\gt, p(\gs)).
\]

\paragraph{The chain $\gO$.}
Suppose $M$ is a $\zz/2$-valued $k$-cycle on $L$.
Identify $M$ with its support (i.e., $M$ is identified with the subcomplex which is the union of those $k$-simplices $\gs$ which have nonzero coefficient in $M$.) Choose a $k$-simplex $\gD\in M$ with vertices $v_0,\dots, v_k$.

Let $v_i^{\pm}$ denote the two vertices in $OM$ lying above $v_i$.
Let $D$ be the full subcomplex of $\OL$ containing $M^-$ and the doubled vertices $v_0^\pm,\dots, v_k^\pm$ of $\gD$.
We say that $D$ is $M$ \emph{doubled over the simplex} $\gD$.
Define a chain $\gO\in C_{2k}(\cac(D);\zz/2)$ by declaring the $2k$-cell $[\gs,\gt]$ of $\cac(D)$ to be in $\gO$ if and only if
\begin{itemize}
	\item $\sigma \cap \tau = \emptyset$, and
	\item $\Delta^0 \subset p(\sigma) \cup p(\tau)$.
	(Here $\gD^0$ denotes the $0$-skeleton of $\gD$.)
\end{itemize}

\paragraph{The cocycle $\mu$.}
Fix a total ordering \( v_0 < v_1 < \dots < v_n \) on the vertex set of $L$ and extend that to a total ordering on the vertex set of $\OL$ by defining
\[
	v_0^- < v_0^+< v_1^- < v_1^+< \cdots < v_n^- < v_n^+.
\]
Define a top degree cocycle $\mu$ in $Z^{2k}(\OL \times L; \zz )$, $k=\dim L$, by
\begin{equation*}
	\mu ( \gs,b ) =
	\begin{cases}
		1, &\text{if $v_0 \le w_0 < v_{1} \le \cdots < v_k \le w_k$,}\\
		0, &\text{otherwise.}
	\end{cases}
\end{equation*}

Here $L$ denotes the $-$ copy of $L$ in $\OL$, i.e., each $w_i$ has sign $-$.
If $\mu ( \gs,b ) = 1$, we say that $\gs$ and $b$ mesh \emph{nonstrictly}.

In the next four lemmas we describe some properties of $s$, $\gO$ and $\mu$.
\begin{lemma}\label{l:pushforward}
	$s_{*}(\gO)=O\gD \times M$ in $C_{2k}(\OL \times L; \zz/2 )$.
\end{lemma}
\begin{proof}
	We compute (modulo 2) the number of times $S$ that $(\gs,b)$ appears on the right hand side of the formula
	\[
		s_{*}([\gs,\gt])=(\gs,p(\gt))+(\gt, p(\gs)),
	\]
	as $[\gs,\gt]$ varies over $\gO$.
	If $\gD \not\subset p(\gs) \cup b $ or if $b\notin M$, then it follows from the definition of $\gO$ that $S=0$.
	So assume $\gD \subset p(\gs) \cup b$ and $b\in M$.
	The preimage of $(\gs,b)$ is the set of pairs $[\gs,\gt]$ in $\gO$ such that $p(\gt) = b$.
	Since $(\gs,p(\gt)) \ne (\tau,p(\gs))$ for such pairs, each such pair contributes once to $S$; so $S$ is the number of $\gt$ disjoint from $\gs$ such that $p(\gt) = b$.
	It follows that
	\[
		S=2^{\vert (b\cap \gD) - p(\gs) \vert }.
	\]
	So, modulo $2$, $S=1$ if and only if $ (b\cap \gD) \subset p(\gs)$.
	Combining this with $\gD \subset p(\gs) \cup b$, we get that $S=1$ if and only if $p(\gs)=\gD$.
	Since the condition $p(\gs)=\gD$ is equivalent to $\gs \in O\gD$, the lemma follows.
\end{proof}
\begin{lemma}\label{l:pullback}
	$s^{*}(\mu)=\nu$ in $C^{2k}(\cac(\OL); \zz).$
\end{lemma}
\begin{proof}
	Let $[\gs,\gt]$ be an oriented cell in $\cac(\OL)$.
	We can assume that $v_{0} < w_{0}$.
	This implies $w_{0} > p(v_{0})$ and therefore, that $\mu(\gt,p(\gs))=0$.
	So, we need to check that $\mu(\gs,p(\gt))=\nu([\gs, \gt])$.
	There are two cases to consider:
	\begin{description}
		\item[$\gs$ and $\gt$ mesh.] We have $v_0 < w_0 < v_{1} < \cdots < v_k < w_k$.
		Applying $p$ to the $w$ terms gives $v_0 \le p(w_0) < v_{1} \le \cdots < v_k \le p(w_k)$, and $\mu(\gs,p(\gt))=1$ as required.
		\item[$\gs$ and $\gt$ do not mesh.] This means that in the meshing string at least one of the inequalities (but not the first one) is reversed.
		There are two cases:
		
		If $v_{i} > w_{i}$, then $v_{i} >p(w_{i})$, so $\mu(\gs,p(\gt))=0$.
		
		If $w_{i} > v_{i+1}$, then $p(w_{i}) \ge v_{i+1}$, so $\mu(\gs,p(\gt))=0$.
	\end{description}
\end{proof}
We say that $(M,\gD)$ \emph{satisfies the $*$-condition} if
\[
	\tag{$*$} \text{For all $\gs,\gt \in M$ with $\Delta^0 \subset \sigma \cup \tau$ we have $\sigma \cap \tau \subset \Delta$.}
\]
\begin{lemma}\label{l:omegacycle}
	Suppose $(M,\gD)$ satisfies the $*$-condition.
	Then $\gO$ is a cycle.
\end{lemma}
\begin{proof}
	Let $\sigma$ be a $k$-simplex and $\alpha$ be a $(k-1)$-simplex in $D$.
	The argument divides into two cases.
	\begin{enumerate}[(i)]		
		\item $\Delta^0 \not\subset p(\sigma) \cup p(\alpha)$.
		In this case there are either two or zero choices for a vertex $x$ such that $(\alpha \ast x) \in D$ and $[\sigma, \alpha \ast x] \in \gO$, since $p(x)$ must be the missing vertex of $\gD$.
		\item $\Delta^0 \subset p(\sigma) \cup p(\alpha)$.
		This condition implies that if $y\in \gD$ and $( p( \ga ) * y ) \in M$, then exactly one of two preimages $x\in p^{-1}(y)$ is not a vertex of $\gs$.
		Moreover, it follows from our assumption that if $x$ is a vertex of $\gs$ and $( \ga * x ) \in D$, then $p(x) \in \gD$.
		It follows that $p$ restricts to a bijection from the set $\{ x\in D^{0}-\gs^{0} \ | \ \ga*x \in D\}$ to the set $\{y \in M^{0} \ |\ p(\ga)*y \in M\} $.
		Since $M$ is a cycle, the range has even cardinality, so there are an even number of $[\sigma, \alpha \ast x] \in \gO$.
	\end{enumerate}
\end{proof}
\begin{lemma}\label{l:muval}
	Suppose $M \in Z_{\dim L}(L;\zz/2)$ is a cycle of top degree.
	Then $\mu(O\gD \times M)=1$
\end{lemma}
\begin{proof}
	If a pair $(\gs, b)$ in $O\gD \times M$ meshes nonstrictly, then so does any pair $(\gs', b)$ where $\gs'$ agrees with $\gs$ over $b\cap \gD$.
	Since for $b\neq \gD$ there are an even number of such $\gs'$, the only contribution comes from the cell $(\gD,\gD)$.
\end{proof}

\section{Proof of the Main Theorem}

\subsection{The case where \texorpdfstring{$H_k(L;\zz/2)= 0$}{H\_k(L;Z/2)=0}}

We prove a somewhat stronger version of statement (2) of the Main Theorem.
\begin{theorem}	
	Suppose $L$ is a $k$-dimensional complex.
	If $H_k(L;\zz/2)=0$, then $\vk^{2k}(\OL)=0$.
	So, for $k\neq 2$, $\embdim \OL\leq 2k$.
	If $L$ is a flag complex, then $\actdim A_L\leq 2k+1$.
\end{theorem}
\begin{proof}
	If $\vk^{2k}(\OL)\neq 0$, then by Lemma~\ref{l:pullback} $s^*([\mu])$ is an element of order 2 in $H^{2k}(\cac(\OL);\zz)$.
	So, the subgroup generated by $[\mu] \in H^{2k}(\OL \times L;\zz)$ contains an index $2$ subgroup.
	Thus, $H^{2k}(\OL \times L;\zz)$ has either a $\zz/2^{r}$ or $\zz$ summand.
	By the Künneth Formula, $H^{2k}(\OL \times L; \zz)=H^{k}(OL; H^{k}(L; \zz))$.
	Hence, $H^{k}(L;\zz)$ has either a $\zz/2^{r}$ or $\zz$ summand.
	So $H_k(L;\zz/2)=\Hom(H^{k}(L;\zz), \zz/2) \neq 0$, a contradiction.
	Since for $k\neq 2$, $\vk^{2k}$ is the complete obstruction, $\OL$ embeds in $S^{2k}$.
	So for $L$ a flag complex, Proposition~\ref{p:emb-act} gives $\actdim A_{L} \leq 2k+1$.
\end{proof}

\subsection{The case where \texorpdfstring{$H_k(L;\zz/2)\neq 0$}{H\_k(L;Z/2)≠0}}
\begin{theorem}\label{t:neq0}
	Let $L$ be a complex and suppose there is a $k$-cycle $M \in Z_k(L;\zz/2)$ and a simplex $\Delta \in M$ so that $(M,\gD)$ satisfies the $*$-condition.
	Then $\vkt^{2k} (\OL) \neq 0$ and $\vkdim \OL \geq 2k$.
\end{theorem}
\begin{proof}
	Applying Lemmas~\ref{l:pushforward}, \ref{l:pullback}, \ref{l:omegacycle}, and~\ref{l:muval} to the $k$-skeleton of $L$ shows that $\vkt^{k}(\OL^{k})$ evaluates nontrivially on the cycle $\gO$.
	Since $\vkt(\OL^{k})$ is the pullback of $\vkt(\OL)$, the result follows.
\end{proof}
\begin{remark}
	If $L$ is a flag complex, then with hypotheses as above, it follows that $\actdim W_{\OL}=\actdim A_L \ge 2k+2$.
	Moreover, since $\obdim \gG$ is a quasi-isometry invariant, any group $\gG$ quasi-isometric to $A_L$ has $\actdim(\gG) \ge 2k+2$.
\end{remark}

Note that if $L$ is a flag complex, then the $*$-condition is automatically satisfied for any top-dimensional cycle, so as a corollary we get statement (1) of the Main Theorem.
\begin{theorem}	
	If $L$ is a $k$-dimensional flag complex and $H_k(L;\zz/2)\neq 0$, then $\vkdim \OL = 2k$ and $\actdim A_L =2k+2= 2\geod A_L$.
\end{theorem}
\begin{remark}\label{r:annoying}
	It is annoying that we need this additional hypothesis of the $*$-condition.
	We conjecture that when $H_l(L;\zz/2)\neq 0$ we can always choose $M$ and $\gD$ to satisfy the $*$-condition.
	
	The following picture illustrates how our argument breaks down.
	Let $M$ be a $2$-cycle whose support contains the four shaded triangles.
	One checks that when we double over $\Delta$, the coefficient of $\partial \gO$ at $\{\sigma,\alpha\}$ is not $0$.
	In fact, if one doubles over $\Delta$ as in the picture, and $L$ contains no smaller $l$-cycles, it turns out that $D$ is not a $2l$-obstructor.
	
	\medskip
	
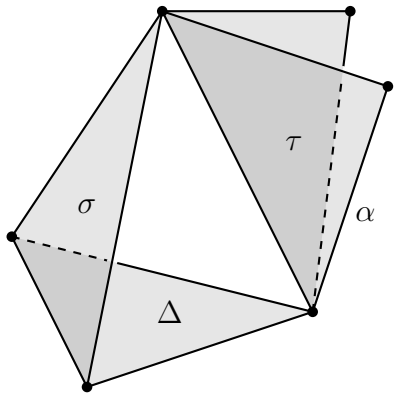
\begin{center}
		\begin{tikzpicture}[thick, scale = 1.0]
			
			\coordinate (a) at (0, 0);
			\coordinate (b) at (1, -2);
			\coordinate (c) at (4,-1);
			\coordinate (d) at (2, 3);
			\coordinate (e) at (5, 2 );
			\coordinate (f) at (4.5, 3);
			
			\draw (b)--(c);
			\path[name path=ac] (a) -- (c);
			\path[name path=bd] (b) -- (d);
			\path[name intersections={of= ac and bd, by=t }];
			\draw[dashed] (a)--(t);
			\draw (t)--(c);
			\fill[white] (t) circle[radius=2pt];
			
			\fill[fill = black, fill opacity = .1] (a)--(b)--(c)--cycle;
			\draw[fill = black, fill opacity = .1] (a)--(b)--(d)--cycle;
			
			\draw (d)--(f);
			\path[name path=cf] (c) -- (f);
			\path[name path=de] (d) -- (e);
			\path[name intersections={of= cf and de, by=x }];
			\draw[dashed] (c)--(x);
			\draw (x)--(f);
			\fill[white] (x) circle[radius=2pt];
			
			\fill[fill = black, fill opacity = .1] (c)--(d)--(f)--cycle;
			
			\draw[fill = black, fill opacity = .1] (c)--(d)--(e)--cycle;

			\fill (a) circle[radius=2pt];
			\fill (b) circle[radius=2pt];
			\fill (c) circle[radius=2pt];
			\fill (d) circle[radius=2pt];
			\fill (e) circle[radius=2pt];
			\fill (f) circle[radius=2pt];
			
			\node at (2.1, -1) {$\Delta$};
			\node at (1, .4) {$\sigma$};
			\node at (3.75, 1.25) {$\tau$};
			\node at (4.7, .3) {$\alpha$};
		\end{tikzpicture}
	\end{center}
\end{remark}

\section{Bounds on the van Kampen dimension \texorpdfstring{of $\boldsymbol{\OL}$}{of OL}} 
The following basic lemma is proved in \cite{bkk}.
\begin{joinlemma}[cf.~\cite{bkk}*{p.~224}] $\vkdim(K_1*K_2) = \vkdim(K_1) + \vkdim(K_2) + 2$.
\end{joinlemma}

Given a simplex $\gs\in L$, denote the link (resp.
closed star) of $\gs$ in $L$ by $\lk(\gs)$ (resp.
$\st(\gs)$).
We then have $\st(\gs)=\lk(\gs)* \gs$.
Since octahedralization commutes with taking joins and since $\vkdim S^{\dim\gs}=\dim\gs -1$, the Join Lemma gives
\begin{align*}
	\vkdim(O\st(\gs))&=\vkdim (O\lk(\gs)) +\vkdim (O\gs) +2 \\
	&=\vkdim (O\lk(\gs)) +\dim \gs +1.
\end{align*}
Sometimes we can use this observation to determine $\actdim A_L$.
For example, for a flag complex $L$, if $H_k(L;\zz/2)$ vanishes in the top degree and $H_{k-1}(\lk(v);\zz/2)\neq 0$ for some vertex $v\in L$, then, by Theorem~\ref{t:neq0}, $\vkdim O\lk(v)=2k-2$ and so, $\vkdim \OL= \vkdim O\st(v)=2k-1$.
Hence, $\embdim \OL=2k$ and $\actdim A_{L}=2k+1$.

\section{The Action Dimension Conjecture}

Suppose $\gG$ acts properly and cocompactly on a finite dimensional, acyclic CW complex $Y$.
The $\ell^2$-Betti number $\lb_i(\gG)$ is then defined as the von Neumann dimension of the $i^{th}$ reduced $\ell^2$-homology group of $Y$.
The \emph{$\ell^2$-dimension of $\gG$}, denoted $\ltdim \gG$, is defined by
\[
	\ltdim \gG:=\max\{i\mid \lb_i(\gG)\neq 0\}.
\]

The most well known conjecture concerning $\ell^2$-homology is the following.
\begin{Sconj}
	If $\gG$ is a $n$-dimensional Poincaré duality group, then $\lb_i(\gG)$ vanishes for $i \ne n/2$.
\end{Sconj}

In \cite{do01}*{Conjecture 0.8} the second and third authors conjectured the following.
\begin{Aconj}	
	$\actdim \gG \geq 2\ltdim \gG.$
\end{Aconj}
\begin{remark}
	In~\cite{os14} the third and fourth authors show that the Singer Conjecture implies the Action Dimension conjecture in many cases.
	In particular, the conjectures are equivalent if one restricts to actions on manifolds that are equivariantly $PL$-triangulated.
\end{remark}
\begin{examples}
	Here are some examples when the conjecture holds.
	\begin{enumerate1}
		\item		
		If all $\ell^2$-Betti numbers of $\gG$ are $0$ (e.g., if $\gG$ contains an infinite amenable normal subgroup), then the Action Dimension Conjecture for $\gG$ holds trivially.
		\item\label{i:singer}
		If $\gG$ is the fundamental group of a closed aspherical manifold $M^n$ and the Singer Conjecture holds for $M^n$, then the Action Dimension Conjecture holds for $\gG$.
		\item\label{i:lattice}
		If $\gG$ is a lattice in a semisimple Lie group without compact factors, then the conjecture holds for $\gG$.
		\item\label{i:mcg}
		If $\gG$ is the mapping class group of a surface with marked points or punctures, then the conjecture holds for $\gG$.
		\item\label{i:2gd}
		If $\actdim(\gG) = 2\geod(\gG)$, then the conjecture holds for $\gG$.
		(For example, it is proved in \cite{bkk} that this is true for $\gG=Out(F_n)$.)
	\end{enumerate1}
	\begin{proof}[Sketch of proofs] We indicate proofs for examples (2) through (5).
		
		\ref{i:singer} The Singer Conjecture for $M^n$ asserts that $\lb_i(\gG)=0$ for $i\neq n/2$.
		Since $\actdim(\gG) = n$, this implies the Action Dimension conjecture.
		
		\ref{i:lattice} By using the technique of \cite{bkk}, Bestvina and Feighn proved in \cite{BF} that the action dimension of the lattice $\gG$ is the dimension $n$ of the corresponding symmetric space $G/K$.
		Using square summable differentiable forms one sees that the only reduced $\ell^2$-cohomology group of $G/K$ is in degree $n/2$.
		Finally, it follows from a result of Cheeger and Gromov~\cite{cg85} that a $\gG$-stable, cocompact, contractible submanifold of $G/K$ has the same $\ell^2$-cohomology groups as $G/K$; hence, the only possible nonzero $\ell^2$-Betti number of $\gG$ lies in degree $n/2$.
		
		\ref{i:mcg} The argument when $\gG$ is a mapping class group is similar.
		Using \cite{bkk}, Despotović~\cite{d06a} showed that $\actdim \gG =\dim \mathcal {T}$, where $\mathcal{T}$ is the appropriate Teichmüller space.
		McMullen \cite{m00} showed that $\mathcal{T}$ admits a K\"ahler hyperbolic metric, and Gromov~\cite{g91} showed that this implies that $\mathcal{T}$ has reduced $\ell^2$-cohomology only in the middle dimension.
		Finally, the same theorem of Cheeger--Gromov shows that $\gG$ has nonzero $\ell^2$-Betti number only in the middle dimension.
		
		\ref{i:2gd} This is immediate from the inequality $\ltdim(G) \le \geod(G)$.
	\end{proof}
\end{examples}

The $\ell^2$-Betti numbers of any RAAG $A_L$ were calculated in \cite{dl03}.
(Actually, the result of \cite{dl03} stated below is valid for the universal cover of the Salvetti complex of any Artin group $A_L$, where $L$ denotes the nerve of the associated Coxeter group.)
\begin{theorem}[cf.~Davis--Leary~\cite{dl03}]
	
	Suppose $\redb_i(L)$ denotes the ordinary reduced Betti number, $\dim_{\QQ} \redh_i(L;\QQ)$.
	Then
	\[
		\btwo_{i+1} (A_L)=\redb_i(L).
	\]
\end{theorem}
\begin{corollary}\label{c:ltdim}
	The $\ell^2$-dimension of a RAAG $A_L$ is given by
	\[
		\ltdim A_L=1+\max\{i\mid \redb_i(L) \neq 0\}.
	\]
\end{corollary}
\begin{theorem}\label{t:actdimconj}
	Suppose the $k$-dimensional flag complex $L$ satisfies $H_k(L;\zz/2)\neq 0$.
	Then the Action Dimension Conjecture holds for $A_L$.
\end{theorem}
\begin{proof}
	
	Since $\geod A_L=k+1$, $\ltdim A_L \leq k+1$.
	By Theorem~\ref{t:neq0},
	\[
		\actdim A_L =2k+2 \geq 2 \ltdim A_L .
	\]
\end{proof}
\begin{remark}
By the Universal Coefficient Theorem, if $H_i(L;\QQ)\neq 0$, then $H_i(L;\zz/2)\neq 0$.  So, if we could remove the annoying hypothesis in Theorem~\ref{t:neq0} (concerning the $*$-condition), then Corollary~\ref{c:ltdim} would imply that the Action Dimension Conjecture holds for all RAAGs (cf.~Remark~\ref{r:annoying}).
\end{remark}

Grigori Avramidi, University of Utah, Department of Mathematics, 155 S. 1400 E., Salt Lake City, UT 84112-0090, \url{gavramid@math.utah.edu}

Michael W. Davis, Department of Mathematics, The Ohio State University, 231 W. 18th Ave., Columbus Ohio 43210, \url{davis.12@math.osu.edu}

Boris Okun, University of Wisconsin-Milwaukee, Department of Mathematical Sciences, PO Box 413, Milwaukee, WI 53201-0413, \url{okun@uwm.edu}

Kevin Schreve, University of Wisconsin-Milwaukee, Department of Mathematical Sciences, PO Box 413, Milwaukee, WI 53201-0413, \url{kschreve@uwm.edu}
\end{document}